\documentclass[10pt]{amsart}
\usepackage{amsmath,amssymb,amsthm}
\usepackage{mathtools}
\usepackage{stmaryrd} 
\usepackage[inline]{enumitem}
\usepackage{graphicx,setspace}
\usepackage{tikz}
\usetikzlibrary{arrows}
\usepackage{wasysym}
\usepackage[makeroom]{cancel}
\usepackage{fullpage}
\usepackage{comment}
\usepackage{enumitem}
\usepackage{caption}
\usepackage{relsize}
\allowdisplaybreaks
\newcommand\blfootnote[1]{%
  \begingroup
  \renewcommand\thefootnote{}\footnote{#1}%
  \addtocounter{footnote}{-1}%
  \endgroup
}

\providecommand{\syl}[2]{{Syl}_{#1}(#2)}

\theoremstyle{plain}
\newtheorem{theorem}{Theorem}[section]
\newtheorem{Proposition }[theorem]{Proposition }
\newtheorem{algorithm}[theorem]{Algorithm}

\newtheorem{lemma}[theorem]{Lemma}
\newtheorem{corollary}[theorem]{Corollary}

\newtheorem*{claim*}{Claim}
\theoremstyle{definition}
\newtheorem{remark}[theorem]{Remark}
\newtheorem{definition}[theorem]{Definition}
\newtheorem{example}[theorem]{Example}

\renewcommand{\c}[1]{\overline{#1}}

\def\o{\oplus}

\def\Z{\mathbb{Z}}

\def\lcm{\text{lcm}}
\providecommand{\Zn}[1]{\Z/#1\Z}
\def\T{\tau}
\def\Tn{\tau_{(n)}}
\def\Tp{\tau_{(p)}}
\def\Tnp{\tau'_{(n)}}
\def\G{U'(n)}
\def\Gp{U'(p)}

\begin{document}
\title{}
\begin{center}\huge{On the Characterization of $\tau_{(n)}$-Atoms}\end{center}

\author[A. Hern\'{a}ndez-Espiet]{A. \ Hern\'{a}ndez-Espiet}
\address{Department of Mathematics\\ University of Puerto Rico\\
Mayagüez, 00682, Puerto Rico}
\email{andre.hernandez@upr.edu}

\author[R. M. Ortiz-Albino]{R.M.\ Ortiz-Albino}
\address{Department of Mathematics\\ University of Puerto Rico\\
Mayagüez, 00682, Puerto Rico}
\email{reyes.ortiz@upr.edu}

\maketitle

\begin{abstract}
In 2011, Anderson and Frazier define the concept of $\tau_{(n)}$-factorization, where $\tau_{(n)}$ is a restriction of the modulo $n$ equivalence relation. These relations have been worked mostly for small values of $n$. However, it is sometimes difficult to extend findings to larger values of $n$. One of these problems is finding $\tau_{(n)}$-irreducible elements or $\tau_{(n)}$-atoms in order to characterize elements that have a $\tau_{(n)}$-factorization in $\tau_{(n)}$-atoms. The $\tau_{(n)}$-irreducible elements are well known for $n=0,1,2,3,4,5,6,8,10,12$. However, the problem of determining the $\tau_{(n)}$-atoms becomes much more difficult the larger $n$ is. In this work, we present an algorithm to construct families of $\tau_{(n)}$-atoms. It is shown that the algorithm terminates in finitely many steps when $n$ is the safe prime associated to a Sophie Germain prime.
\end{abstract}
\section{Introduction}\label{intro}

\blfootnote{\textit{2010 Mathematics Subject Classification}. 13A05, 11A99, 20K01\\ \textit{Key words and phrases}. $\Tn$-Factorizations, Factorizations over Integral Domains, Generalized Factorizations}

Anderson and Frazier developed the theory of $\tau$-factorizations \cite{anderson_frazier_2011}, or $\tau$-products on integral domains, where $\tau$ is a symmetric relation that determines which elements are allowed to be multiplied. This concept can be visualized as the study of a restriction to the multiplicative operation. That is, two nonzero nonunit elements are allowed to be multiplied if and only if they are related with respect to the symmetric relation $\tau$. Formally, let $\tau$ be  a symmetric relation on the nonzero nonunit elements of an integral domain $D$ (denoted by $D^\#$). We say that $x\in D^\#$ has a $\tau$-factorization if $x=\lambda x_1*** x_n$, where $\lambda$ is a unit and for any $i\neq j$, $x_i\tau x_j$. In such case, we also say that each $x_i$ is a $\tau$-factor of $x$ (denoted by $x_i|_\tau x$). Notice that, $x=x$ and $x=\lambda(\lambda^{-1}x)$ are both (vacuously) $\tau$-factorizations, called the trivial ones. In order to distinguish between an usual product from a $\tau$-product we will denote $a_1\cdot a_2\cdots a_n$ (respectively, $a_1*a_2***a_n$) the usual product (respectively, a $\tau$-product) of the elements $a_1,a_2,\cdots,a_n$. An element $x\in D^\#$ whose only $\tau$-factorizations are the trivial ones are called $\tau$-atoms or $\tau$-irreducible elements. For example, if $\tau=S\times S$ where $S$ is a subset of the nonzero nonunit elements, then the $\tau$-products are the usual products of elements in $S$. The theory of $\tau$-factorization was also called the theory of generalized factorizations. This is because one can consider the relation $\tau=S\times S$, where $S$ is any specific subset of $D^\#$. As an example, take $S$ to be the set of irreducible elements (resp. primes, primal elements, etc), hence the $\tau$-products would be factorizations into irreducible elements (resp. into primes, primal elements, etc). Also, if $S=D^\#$, then the $\tau$-products and the usual products of $D^\#$ coincide.\\

In \cite{anderson_frazier_2011}, the authors presented results based on a classification of types of relations. Two of these types are divisive (if $x\tau y$ and $x'|x$, then $x'\tau y$) and associated-preserving (if $x\tau y$ and $x'\tau x$, $x'\tau y$) relations. Associated-preserving relations are well behaved compared to other relations, as can be seen in \cite{anderson_frazier_2011}, and allow to omit the unit multiple in front (if $\lambda x_1*** x_n$ is a $\tau$-factorization, so is $x_1*** x_{i-1}*(\lambda x_i)* x_{i+1}*** x_n$). Divisive relations allow to admit $\tau$-refinements; that is, whenever $x=\lambda x_1***x_n$ and $x_i=y_1*** y_m$ are $\tau$-factorizations, so is $x=\lambda x_1*** x_{i-1}* y_1*** y_m*x_{i+1}*** y_m$. It is important to note that if a relation is divisive, then it is associated-preserving. Most of the resutls in \cite{anderson_frazier_2011} assumed relations to be divisive.\\

On the other hand, to avoid studying the usual product of the structure we must avoid relations that are both reflexive and divisive because, by reflexivity, for any $x,y\in D^\#$, $(xy)\tau(xy)$. On the other hand, by the divisivity of $\tau$, $x\tau y$. This motivated Ortiz and Serna \cite{serna} to study the behavior of the $\tau$-factorizations, when $\tau$ is an equivalence relation (aside of their historical impotance). Their main results were based on unital equivalence relations (that is, equivalence relations with the following property: if $x\tau y$, then for any unit $\tau$, $(\lambda x)\tau(\lambda y)$). If $\tau$ is an equivalence unital relation, one may assume that the equivalence relation is also associated-preserving. If $\tau'$ is the associated-preserving closure of $\tau$, then $x$ has a $\tau$-factorization (respectively, $x|_\tau y$, $x$ is a $\tau$-atom) if and only if $x$ has a $\tau'$-factorization (respectively, $x|_{\tau'}y$, $x$ is a $\tau'$-atom). In other words, by Theorem 4.17 of \cite{serna}, $x_1***x_n$ is a $\T'$-factorization if and only if there are units $\lambda,\lambda_1,\ldots,\lambda_n$, such that $\lambda(\lambda_1x_1)***(\lambda_n x_n)$ is a $\T$-factorization. See \cite{serna} to see other properties that are preserved with respect to $\tau$ and $\tau'$.\\

As a special case, the authors of \cite{serna} consider the relation $\tau_{(n)}$ defined on $\Z^\#$ by Anderson and Frazier \cite{anderson_frazier_2011}, and followed by Hamon \cite{hamon}. The relation $\tau_n=\{(x,y)|x-y\in (n)\}$, which can be seen as the equivalence relation modulo $n$ restricted to $\Z^\#$. Hamon gave a characterization for which $n$ it is true that every element of $\Z^\#$ has a $\tau_{(n)}$-factorization into $\tau_{(n)}$-atoms. Hamon's results include $\{0,1,2,3,4,5,6,8,10,12\}$. In 2012, Juett \cite{juett2014} found out that $12$ does not satisfy such property. It is important to note that these results did not require knowing which elements are $\Tn$-atoms.\\

An attempt to find out $\tau_{(n)}$-atoms for any $n$ was given by Lanterman \cite{lanterman} in 2012 (presented at JMM 2013). In his paper, there is no information about the technique or method used to reduce the problem of finding the $\tau_{(n)}$-atoms to a finite number of cases by hand. This research shows that his technique only applies when both $n$ and $\frac{n-1}{2}$ are primes (for more details, see Section \ref{prelim}). A second attempt of finding $\tau_{(n)}$-atoms was done by Molina in \cite{molina}, by using the formula of the number of $\tau_{(n)}$-factors for any positive integer. He was able to provide the explicit forms of the $\tau_{(n)}$-atoms for $n=8,10,12$. The technique used in this work is more general and easier than previous attempts.

\section{Preliminaries}\label{prelim}

Serna's result linking $\Tn$-factorizations to $\Tnp$-factorizations tells that to determine whether or not a number has a $\Tn$-factorization, that we could work with $\Tnp$-factorizations instead. In general, the results obtained in this report were obtained using the equivalence relation $\Tnp$. We are now in a position to define $\G$.\\

For $n\geq 2$, let $U(n)=\{m\in\Z/n\Z|\gcd(n,m)=1\}$, the multiplicative group of units in $\Z/n\Z$ with order $\varphi(n)$, where $\varphi(n)$ is Euler's totient function, as in \cite{gallian_2013}. We define $\G$ for $n\geq 3$ to be the set $\{\c{k}|k\in U(n)\}$, where 

$$\c{k}=k\cup -k=\{x\in\Z|x\equiv k\pmod n\text{ or }x\equiv -k\pmod n \}.$$ 
\\
Since $U'(2)=U(2)$, we focus on $n\geq 3$. We shall see that $\G$ is a group with operation, using multiplicative notation, $\c{a}\cdot \c{b}=\c{ab}$, including some of its properties.

\begin{Proposition }
$(\G,\cdot)$ is an abelian group of order $\frac{\varphi(n)}{2}$, isomorphic to $U(n)/\{\pm 1\}$.
\end{Proposition }

\begin{proof}
Let $(\c{a},\c{b})=(\c{c},\c{d})$, for $\c{a},\c{b},\c{c},\c{d}\in\G$. By the defintion of $\G$, $n|a-c$ or $n|a+c$ and $n|b-d$ or $n|b+d$. So, there exist $\alpha,\beta\in\Z$ such that $n\alpha=a\pm c$ and $n\beta=b\pm d$, where occurrences of $\pm$ are possibly exclusive and are independent. We have that

$$ab=(\pm c+n\alpha)(\pm d+n\beta)=\pm cd+n(\pm\alpha d\pm c\beta +n\alpha\beta).$$

Regardless of the choices for each $\pm$, we obtain that $n|ab+cd\text{ or } n|ab-cd$. In other words, by the definition of $\G$, $\c{ab}=\c{bd}$. Therefore, the operation defined on $\G$ is well defined.\\

Let $\c{a},\c{b}\in\G$. Since $ab\in U(n)$, $\c{ab}\in\G$. Both associativity and commutativity are inherited from $U(n)$. If $\c{a}\in\G$, then $\c{a}\cdot \c{1}=\c{a}=\c{1}\cdot \c{a}$. This shows that $\c{1}$ is the identity of $\G$. Let $\c{a}\in\G$. Let $b$ be the multiplicative inverse of $a$ in $U(n)$. We then have that $\c{a}\cdot\c{b}=\c{ab}=\c{1}=\c{ba}=\c{b}\cdot\c{a}$, so that $\c{b}$ is the multiplicative inverse of $\c{a}$ in $\G$.\\

Define the map $f$ from $U(n)$ to $\G$ as $f(k)=\c{k}$. We shall prove that $f$ is a homomorphism. If $k\equiv k'\pmod n$, then $\c{k}=\{x\in U(n)|x\equiv\pm k\pmod n\}$ and $\c{k'}=\{x\in U(n)|x\equiv\pm k'\pmod n\}$ are clearly the same. Therefore, $f$ is well defined. If $\c{k}\in\G$, $\gcd(k,n)=1$. Hence, $k\in f^{-1}(\c{k})$, making $f$ surjective. Also, if $a,b\in U(n)$, then $f(ab)=\c{ab}=\c{a}\cdot\c{b}=f(a)f(b)$, so that $f$ is a homomorphism. If $x\in\ker(f)$, then $\c{x}=\c{1}$. That is, $x\equiv 1\pmod n$, or $x\equiv -1\pmod n$. So, $\ker(f)=\{\pm 1\}$. By the first isomorphism theorem, $U(n)/\{\pm 1\}\cong\G$, and $|\G|=\frac{\varphi(n)}{2}$.

\end{proof}

As an example, consider $n=11$. We have $U(11)=\{1,2,3,4,5,6,7,8,9,10\}$. Meanwhile, $U'(11)=\{1,2,3,4,5\}$. The following table presents the Cayley table of $U'(11)$, and shows that it is a cyclic group isomorphic to $\Zn{5}$:

\begin{table}[h]
\scalebox{1.5}{
\begin{tabular}{|c||c|c|c|c|c|c|c|c|}
 \hline
 $\cdot$ & $\mathsmaller{\c{1}}$ & $\mathsmaller{\c{2}}$ & $\mathsmaller{\c{3}}$ & $\mathsmaller{\c{4}}$ & $\mathsmaller{\c{5}}$ \\ 
 \hline
 \hline
$\mathsmaller{\c{1}}$ & $\mathsmaller{\c{1}}$ & $\mathsmaller{\c{2}}$ & $\mathsmaller{\c{3}}$ & $\mathsmaller{\c{4}}$ & $\mathsmaller{\c{5}}$\\
\hline
$\mathsmaller{\c{2}}$ & $\mathsmaller{\c{2}}$ & $\mathsmaller{\c{4}}$ & $\mathsmaller{\c{5}}$ & $\mathsmaller{\c{3}}$ & $\mathsmaller{\c{1}}$\\
\hline
$\mathsmaller{\c{3}}$ & $\mathsmaller{\c{3}}$ & $\mathsmaller{\c{5}}$ & $\mathsmaller{\c{2}}$ & $\mathsmaller{\c{1}}$ & $\mathsmaller{\c{4}}$\\
\hline
$\mathsmaller{\c{4}}$ & $\mathsmaller{\c{4}}$ & $\mathsmaller{\c{3}}$ & $\mathsmaller{\c{1}}$ & $\mathsmaller{\c{5}}$ & $\mathsmaller{\c{2}}$\\
\hline
$\mathsmaller{\c{5}}$ & $\mathsmaller{\c{5}}$ & $\mathsmaller{\c{1}}$ & $\mathsmaller{\c{4}}$ & $\mathsmaller{\c{2}}$ & $\mathsmaller{\c{3}}$\\
\hline
\end{tabular}}
\caption{Cayley table for $\G$, $n=11$}
\label{table:1}
\end{table}

\begin{Proposition }
\label{method}
Let $n$ be a composite number. If $p$ and $\frac{p-1}{2}$ are prime, and at least $\frac{p-1}{2}$ primes in $\c{i}\neq \c{1}$ dividing $n$, then $n$ has a $\T_{(p)}$-factorization.
\end{Proposition }

\begin{proof}
Let $p_1,p_2,\ldots,p_k$ be the prime numbers that divide $n$ and are all in $\c{i}$ with $\c{i}\neq \c{1}$. Let $n'=\frac{n}{p_1\cdots p_k}$, or equivalently, $n=p_1\cdots p_kn'$. Since $\Gp$ has prime order, $\Gp$ is cyclic. Not only this, but each nonidentity element in $\Gp$ is a generator of the group. In other words, $\c{i}$ generates $\Gp$. By assumption, we also have that $k\geq\frac{p-1}{2}$. Since $\Gp$ is a group, we have that $\c{n'}$ has an inverse. Since $\c{i}$ is a generator, we have that $(\c{n'})^{-1}=(\c{i})^j$, for some $0<j\leq\frac{p-1}{2}$. If $j=\frac{p-1}{2}$, then $\c{n'}=\c{1}$. This would imply that $p_1***p_{k-1}*(p_kn')$ is a $\Tp'$-factorization of $n$. If $j=\frac{p-1}{2}-1$, we than have that $\c{n'}=\c{i}$, so that $p_1***p_{k-1}*p_k*n'$ is a $\Tp'$-factorization of $n$. Otherwise, $0<j\leq\frac{p-1}{2}-2$. For these cases we have that $k-j\geq 2$. Then, $p_1*p_2*****p_{k-j-1}*(p_{k-j}(p_{k-j+1}\cdots p_kn'))$ is a $\Tp$-factorization of $n$, because $\c{p_{k-j+1}\cdots p_kn'}=\c{1}$. By Theorem 4.17 of \cite{serna}, $n$ has a $\Tp$-factorization.
\end{proof}

The contrapositive of this results yields the first method for determining all the $\Tn$-atoms for certain values of $n$. If we want to find a $\Tn$-atom, it will need to have at most $|\G|$ primes that are equal to each other in $\G$ in its factorization. Noting that whether or not numbers have a $\Tn$-factorization depends only on which elements of $\G$ the dividing primes lie in, and not the actual value that the primes have. This last proposition says that we only have to check a finite amount of numbers before we determine all $\Tn$-atoms for these values of $n$, by checking the product of primes to powers lesser than $\frac{p-1}{2}$. It appears that this is the technique used by Lanterman in \cite{lanterman}. In Section \ref{factorization}, we shall show that this technique does not hold in general when $\frac{p-1}{2}$ is not prime.
\section{Structure of $\G$}

Since Proposition \ref{method} depended on the group structure of $\G$, it will prove useful to understand the group structure of $\G$. First, we shall determine exactly when $\G$ was cyclic.\\

In Proposition \ref{method}, we never used the fact that $p$ has to be a prime. The condition that really mattered was that $\frac{\varphi(n)}{2}$ was prime. A question is, for which $n$ does this happen? Well, the answer is for $n=9,12,18,p,2p$ for $p$ any safe prime associated to a Sophie Germain prime.\\

\begin{Proposition }
\label{Gprime}
If $\frac{\varphi(n)}{2}$ is prime, then $n=9,12,18,p$, or $n=2p$ for any safe prime $p$.
\end{Proposition }

\begin{proof}
Let $n$ be a prime power. We have that $\frac{\varphi(n)}{2}=\frac{1}{2}p^{a-1}(p-1)$. The only way that $n$ is not prime in this case is if $p^{a-1}$ is not $1$. This forces $\frac{p-1}{2}$ to be $1$, if $p$ is odd; or $n=4$, if $p$ is even. The first case then forces $n=9$.\\

Let $n$ not be squarefree, but not be a prime power. A similar analysis as in the previous case implies that the only numbers that satisfy the conditions of the hypothesis are $n=12,18$.\\

Now, let $n$ be squarefree, but not be a prime or $2$ times an odd prime. Let us observe when $n$ is a product of $2$ odd primes, $p$ and $q$. We then have that $\frac{1}{2}\varphi(pq)=\frac{(p-1)(q-1)}{2}$. With the exception of $p,q=3$, however, $p-1,q-1$ both have more than $1$ factor. This would tell us that $(p-1)(q-1)$ has at least $4$ factors so that $\frac{(p-1)(q-1)}{2}$ would not be prime. If $p=3$, we still obtain that $(p-1)(q-1)$ has at least $3$ factors so that, again, $\frac{(p-1)(q-1)}{2}$ is not prime. This shows us that this case produces no numbers that satisfy the hypothesis. For the same reason, if $n$ is odd, squarefree, and more than two primes divide it then we obtain no numbers that satisfy the hypothesis. By the same explanation, we have that if $n$ is even, squarefree, and more than three primes divide it, then such numbers do not satisfy the hypothesis.\\

This only leaves us with $n=p$ or $n=2p$, where $p$ is a prime. This case follows directly from the definition of Sophie Germain primes. This concludes our proof.
\end{proof}

\begin{lemma}
\label{uncyclic}
If $U(n)$ is cyclic, then $\G$ is cyclic.
\end{lemma}

\begin{proof}
Since $\G\cong U(n)/\{\pm 1\}$, the result follows from how all quotients of cyclic groups are cyclic.
\end{proof}

By the Primitive Root Theorem in \cite{kumanduri_romero_1997}, Lemma \ref{uncyclic} implies that $n=2,4,p^k,2p^k$ (for any odd prime $p$) are values for which $\G$ is cyclic. It turns out that these are not the only values of $n$ for which $\G$ is cyclic. The next propositions will help us develop a criterion to find the other values of $n$ such that $\G$ is cyclic.

\begin{Proposition }
\label{key1}
Let $x\in U(n)$ with $|x|=\frac{\varphi(n)}{2}$. If $x^i\equiv -1\pmod n$ then $i=\frac{\varphi(n)}{4}$.
\end{Proposition }

\begin{proof}
If $x^i\equiv -1\pmod n$, then $1<i<\frac{\varphi(n)}{2}$, due to the order of $x$. We then have that $x^{2i}\equiv 1\pmod n$. By Lagrange's Theorem in \cite{gallian_2013}, $\frac{\varphi(n)}{2}$ divides $2i$. This means $2i=\frac{\varphi(n)}{2}k$ for some $k$. Since $1<i<\frac{\varphi(n)}{2}$, we have $2<2i<\varphi(n)$. This forces $k=1$, which implies that $2i=\frac{\varphi(n)}{2}$ so that $i=\frac{\varphi(n)}{4}$, as we wanted to show.
\end{proof}

Note that $\c{x}^k=\c{1}$ if and only if $x^k\equiv\pm 1\pmod n$. Hence, if $x^k\not\equiv\pm 1\pmod n$ for $1\leq k<\frac{\varphi(n)}{2}$, then $\c{x}$ is necessarily a generator of $\G$. The previous proposition shows that if there is an $x\in U(n)$ such that $|x|=\frac{\varphi(n)}{2}$ and $x^{\frac{\varphi(n)}{4}}\not\equiv -1\pmod n$, then $\G=\langle\c{x}\rangle$.\\

We shall now see how the group structure of $U(n)$ gets restricted by having an element of order $\frac{\varphi(n)}{2}$.

\begin{Proposition }
\label{decomp}
If $U(n)$ is not cyclic and has an element of order $\frac{\varphi(n)}{2}$ then $U(n)\cong(\Zn{n_1})\times(\Zn{n_2})$, where $\gcd(n_1,n_2)=2$.
\end{Proposition }

\begin{proof}
By Theorem 8.3 of \cite{gallian_2013}, $U(n)$ is isomorphic to a finite product of cyclic groups with even order. By assumption, we have that the product will have length at least $2$. For the sake of contradiction, assume that the product has length at least $3$, i.e., $U(n)\cong(\Zn{n_1})\times\cdots\times(\Zn{n_k})$ with $k\geq 3$. This means that $\varphi(n)=n_1\cdots n_k$. Note that the largest possible order of an element in $U(n)$ will be $\lcm(n_1,n_2,\ldots,n_k)$. Since $|(1,1,\ldots,1)|=\lcm(n_1,\ldots,n_k)$, it has maximum order. The hypothesis implies that $\frac{\varphi(n)}{2}=\frac{n_1\cdots n_k}{2}\leq\lcm(n_1,\ldots,n_k)<n_1\cdots n_k=|U(n)|$. Therefore, $\lcm(n_1,n_2,\ldots,n_k)$ is equal to $\frac{\varphi(n)}{2}$. Note that $|(1,\ldots,1)|=\lcm(n_1,\ldots,n_k)$, so that by the hypothesis, $\lcm(n_1,\ldots,n_k)$ is forced to be $\frac{\varphi(n)}{2}$. Remember that the least common multiple is the product of the maximum prime powers that occur in $n_1,\ldots, n_k$. Also, the maximum power of $2$ only gets counted once. We then have that if $a$ is the maximum power of $2$ dividing any of the $n_i$s that $2^a|\lcm(n_1,n_2,\ldots,n_k)$ and $2^{a+1}\nmid\lcm(n_1,n_2,\ldots,n_k)$. However, $2^{a+2}|\varphi(n)$ (since $k\geq 3$) so that $2^{a+1}|\frac{\varphi(n)}{2}=\lcm(n_1,\ldots,n_k)$, a contradiction to $a$ being the maximum exponent of $2$ that divides $\lcm(n_1,\ldots,n_k)$.  Therefore, the length of the product is exactly $2$.\\

If $\gcd(n_1,n_2)>2$, then

$$\lcm(n_1,n_2)=\frac{n_1n_2}{\gcd(n_1,n_2)}<\frac{n_1n_2}{2}=\frac{\varphi(n)}{2}$$

again, a contradiction to the existence of an element of order $\frac{\varphi(n)}{2}$. Therefore, $U(n)\cong(\Zn{n_1})\times(\Zn{n_2})$, where $\gcd(n_1,n_2)=2$.

\end{proof}

We shall now prove the last supporting proposition necessary to prove our criterion for when $\G$ is cyclic. It shows that the element referred to after the proof of Proposition \ref{key1} exists, proving that if $U(n)$ has an element of order $\frac{\varphi(n)}{2}$, then $\G$ is cyclic.

\begin{Proposition }
\label{key2}
If $U(n)$ is not cyclic and has elements of order $\frac{\varphi(n)}{2}$, then there exists an $x\in U(n)$ such that $|x|=\frac{\varphi(n)}{2}$ and $x^i\not\equiv -1\pmod n$ for $1\leq i<\frac{\varphi(n)}{2}$.
\end{Proposition }

\begin{proof}
By Proposition \ref{decomp}, we have that $U(n)\cong(\Zn{n_1})\times(\Zn{n_2})$ with $\gcd(n_1,n_2)=2$. Note that the gcd restriction guarantees that one of $n_1$ or $n_2$ is divisible by $2$ and not by $4$. Without loss of generality, let $n_1$ be the one dividible by $2$ and not by $4$. This means that $n_1=2(2k_1+1)$ for some $k_1\geq 0$. Then consider two cases. Either $4|n_2$ or $4\nmid n_2$. In the second case there is an isomorphism, $f$, that goes from $(\Zn{2(2k_1+1)})\times(\Zn{2(2k_2+1)})$ to $U(n)$. Note that both $(1,1)$ and $(2,1)$ are elements of order $\frac{\varphi(n)}{2}$. We then see that

\begin{align*}
f\left(\left(\frac{\varphi(n)}{4},\frac{\varphi(n)}{4}\right)\right)=&f\left(\frac{\varphi(n)}{4}(1,1)\right)\\
=&f((1,1))^{\frac{\varphi(n)}{4}}\\
\end{align*}

and

\begin{align*}
f\left(\left(\frac{\varphi(n)}{2},\frac{\varphi(n)}{4}\right)\right)=&f\left(\frac{\varphi(n)}{4}(2,1)\right)\\
=&f((2,1))^{\frac{\varphi(n)}{4}}\\
\end{align*}

Now let $f((1,1))=x$ and $f((2,1))=y$. Observing each first step, $\left(\frac{\varphi(n)}{4},\frac{\varphi(n)}{4}\right)\neq\left(\frac{\varphi(n)}{2},\frac{\varphi(n)}{4}\right)$, since the first has both components odd and the second has the first component even and the second component odd (we are working in an even moduli). This implies that their images under $f$ are different, since $f$ is a bijection. This then means that $x^{\frac{\varphi(n)}{4}}\neq y^{\frac{\varphi(n)}{4}}$. This then means that not both can be congruent to $-1$ modulo $n$. This shows the existence of the requested $x$ in this case, by Proposition \ref{key1}.\\

Now, assume $4|n_2$. This means $U(n)\cong(\Zn{2(2k_1+1)})\times(\Zn{4k_2})$. Note that this implies that $8|\varphi(n)$. If it were true that $x^\frac{\varphi(n)}{4}\equiv -1\pmod n$, then $\left(x^{\frac{\varphi(n)}{8}}\right)^2\equiv -1\pmod n$. In other words, $-1$ is a quadradic residue modulo $n$.\\

We shall derive a contradiction. We consider two cases. First, let $4|n$. We have that $y^2\pmod n\equiv -1\pmod n$ then implies that $4|y^2+1$, which is false by Proposition 9.3.3 of \cite{kumanduri_romero_1997}. Therefore, $4\nmid n$.\\

By Theorem 8.3 of \cite{gallian_2013} the only way left to obtain $U(n)$ isomorphic to a product of length $2$ is if $n$ is the product of two distinct odd prime powers or is the $2$ times the product of two distinct odd prime powers. Denote the two prime powers by $p^s$ and $q^t$. Note that $\varphi(n)=p^{s-1}(p-1)q^{t-1}(q-1)$, and $U(n)\cong(\Zn{p^{s-1}(p-1)})\times(\Zn{q^{t-1}(q-1)})$. However, the previous isomorphism has to match with this isomorphism. Without loss of generality, let $p^{s-1}(p-1)=2(2k_1+1)$. Note that $2\not|p^{s-1}$, so that $2\mid p-1$ and $4\not|p-1$. This implies that $p-1=4c+2$, for some $c$, so that $p\equiv 3\pmod 4$. Now, return to the congruence $y^2\equiv -1\pmod n$. This would then imply that $p^s\mid y^2-1$, which then implies that $p\mid y^2-1$. In other words, $y^2\equiv -1\pmod p$. However, we know from Proposition 9.1.11 of \cite{kumanduri_romero_1997}, that this only happens when $p\equiv 1\pmod 4$, a contradiction.\\ 

In conclusion, if $4|n_2$, $-1$ is not a quadradic residue modulo $n$, which means that $x^\frac{\varphi(n)}{4}\pmod n\equiv -1\pmod n$ has no solution. In other words, every $x$ that has order $\frac{\varphi(n)}{2}$ would satisfy the requirements of the problem.\\

\end{proof}

With these propositions, we can now tell exactly when $\G$ is a cyclic group.

\begin{Proposition }
\label{mresult}
$\G$ is cyclic if and only if $U(n)$ has an element of order $\frac{\varphi(n)}{2}$.
\end{Proposition }

\begin{proof}
($\implies$) Suppose $U(n)$ does not contain an element of order $\frac{\varphi(n)}{2}$. Then, the orders of elements of $U(n)$ are all lesser than $\frac{\varphi(n)}{2}$. If $\c{x}\in\G$, then $x^i\equiv 1\pmod n $ in $U(n)$ for some $0<i<\frac{\varphi(n)}{2}$. Hence, $(\c{x})^i=1$ in $\G$. Note that $i<\frac{\varphi(n)}{2}=|\G|$ so that $\c{x}$ is not a generator for $\G$. Therefore, this is true for all elements of $\G$ so that $\G$ is not cyclic. By the contrapositive of the proposition, we are done.\\

($\impliedby$) If $U(n)$ is cyclic, then we are done by Lemma $\ref{uncyclic}$. Assume otherwise. By Propositions $\ref{key1}$ and $\ref{key2}$, there exists an $x$ of order $\frac{\varphi(n)}{2}$ such that $x^i\not\equiv\pm 1\pmod n$ for $1\leq i<\frac{\varphi(n)}{2}$. This implies that $(\c{x})^i\neq\c{1}$ for $1\leq i<\frac{\varphi(n)}{2}$. Therefore, $\c{x}$ is a generator for $\G$.\\
\end{proof}

Note, however, that we can rephrase the previous proposition by determining exactly when $U(n)$ has an element of order $\frac{\varphi(n)}{2}$.

\begin{Proposition }
$U(n)$ has an element of order $\frac{\varphi(n)}{2}$ if and only if $n=2^ip^k$, or $n=2^jq^kr^t$ for $p$ any prime, $q,r$ odd primes such that $\gcd(q^{k-1}(q-1),r^{t-1}(r-1))=2$, and $i\in\{0,1,2\}$, $j\in\{0,1\}$.
\end{Proposition }

\begin{proof}
$(\impliedby)$ We have that this direction is true by direct application of Theorem 8.3 of \cite{gallian_2013} to each case.\\

$(\implies)$ If $U(n)$ is cyclic, then we see that the conclusion holds true by the primitive root theorem. Assume that $U(n)$ is not cyclic. By Proposition $\ref{decomp}$, we have that $U(n)$ is isomorphic to a product of two cyclic groups. \\

If $n$ is divisible by $8$, then $n$ has to be a power of $2$, by Theorem 8.3 of \cite{gallian_2013}. Otherwise, if $n$ is divisible by $4$, but not $8$, we obtain by Theorem 8.3 of \cite{gallian_2013} that $4$ times an odd prime power is our only option for $n$. Now let $4\nmid n$. If $n$ is even, then $U(\frac{n}{2})\cong U(n)$, since $2|n$ and $4\nmid n$. Therefore, it is enough to consider the case where $n$ is odd. If $n$ contained more than $2$ odd prime powers, then $U(n)$ would be isomorphic to a product of more than $2$ cyclic groups, a contradiction. If it only contained $1$ prime power, we would then contradict that $U(n)$ is not cyclic. Therefore, $n$ would be the product of exactly $2$ prime powers. Finally, by Proposition $\ref{decomp}$ we have that the gcd requirement of the conclusion is satisfied. This concludes the proof.
\end{proof}

We now restate Proposition $\ref{mresult}$.

\begin{corollary}
\label{explicitcyclic}
$\G$ is cyclic if and only if $n=2^ip^k$, or $n=2^jq^kr^t$ for $p$ any prime, $q,r$ odd primes such that $\gcd(q^k(q-1),r^t(r-1))=2$, and $i\in\{0,1,2\}$, $j\in\{0,1\}$.
\end{corollary}

It is interesting to note that in the case that two distinct odd primes divide $n$ that at least one of them is forced to be equivalent to $3$ mod $4$, since otherwise the gcd condition of the previous proposition would not be satisfied.\\



\section{Determining $\Tn$-irreducible elements for $\G$}
\label{factorization}

\begin{definition}
Let $\c{x_1},\ldots,\c{x_k}$ be the distinct elements of $\G$.
\begin{enumerate}
\item An $\Tn$-element form $\alpha=(\c{x_1})^{m_1}\cdots(\c{x_k})^{m_k}$ is the set of all $x\in\Z^\#$ with $\gcd(x,n)=1$ that have exactly $m_i$ primes in the equivalence class $\c{x_i}\in\G$ dividing it, for every $i$. 
\item We say that $x\in\alpha$ is of the $\Tn$-element form $(\c{x_1})^{m_1}\cdots(\c{x_k})^{m_k}$. 
\item The $\Tn$-element form $\alpha$ is said to be equivalent to $\c{a}$ whenever $\c{a}=\c{x}$ for any $x\in\alpha$.
\end{enumerate}

\begin{remark}
Whenever possible, we shall assume that $x_1<x_2<\cdots<x_k$.
\end{remark}

\end{definition}

\begin{example}
For $n=11$ we have $23\cdot 43\cdot 2^3$ is in the $\Tn$-element form $(\c{1})^{2}\cdot(\c{2})^{3}\cdot(\c{3})^{0}\cdot(\c{4})^{0}\cdot(\c{5})^{0}$.
\end{example} 

It follows that if two numbers are of the same $\Tn$-element form, then they either both are $\Tn$-atoms or both have a non trivial $\Tn$-factorization. This is because whether or not a number has a $\Tn$-factorization only depends on which elements of $\G$ the primes lie in and not on the specific values the primes take. This motivates the following definition:

\begin{definition}
A $\Tn$-element form, $\alpha=(\c{x_1})^{m_1}\cdots(\c{x_k})^{m_k}$, is said to be $\Tn$-irreducible if all its elements are $\Tn$-irreducible elements.
\end{definition}

\begin{remark}
Since we only work with one value of $n$ for the equivalence relation $\Tn$ at any specific moment, we will now use the term irreducible form instead of $\Tn$-irreducible form, as the value of $n$ will be clear from context.
\end{remark}

Another property of the element forms is that they partition the nonunit nonzero elements of $\Z$ that are relatively prime to $n$. It follows that by classifying which element forms are irreducible one can determine all of the $\Tn$-irreducible elements that are relatively prime to $n$. It is worth noting that by Dirichlet's theorem on primes in arithmetic progressions that each element form has infinitely many elements, since there are infinitely many choices for primes in each $\c{x}\in\G$. Despite this, for some values of $n$, only finitely many element forms exist, and hence finitely many familites of $\Tn$-irreducible elements exist.\\

We shall now explore some of the properties that these element forms have.

\begin{Proposition }
\label{auto}
Let $\varphi$ be an automorphism of $\G$. We have that the element form $(\c{x_1})^{m_1}\cdots(\c{x_k})^{m_k}$ is irreducible if and only if the element form $\varphi(\c{x_1})^{m_1}\cdots\varphi(\c{x_k})^{m_k}$ is irreducible.
\end{Proposition }
\begin{proof}
$(\implies)$ Let $\varphi$ be any automorphism of $\G$ and $(\c{x_1})^{m_1}\cdots(\c{x_k})^{m_k}$ be the element form under consideration. Let $p_i\in\c{x_i}$ for each $i$ and let $q_i\in\varphi(\c{x_i})$ for each $i$. Then, $x=p_1^{k_1}\cdots p_m^{k_m}$ is in the element form $(\c{x_1})^{m_1}\cdots(\c{x_k})^{m_k}$ and that $y=q_1^{k_1}\cdots q_m^{k_m}$ is in the element form $\varphi(\c{x_1})^{m_1}\cdots\varphi(\c{x_k})^{m_k}$. For the sake of contradiction, assume that the second is not irreducible. Then there is a $\Tn$-factorization of $y$, $a_1***a_t$ given by some rearrangement of the primes in $y$. Let $b_i$ be the product of primes that results from replacing $q_j$ for $p_j$ in $a_i$ for all $j$ and $i$, so that $\varphi(\c{b}_i)=\c{a}_i$. Since $\c{a_i}=\c{a_j}$ for all $i,j$, we have $\varphi(\c{b_i})=\varphi(\c{b_j})$ for all $i,j$. Since $\varphi$ is one to one, this means that $\c{b_i}=\c{b_j}$ for all $i,j$. Therefore, $b_1***b_t=p_1^{k_1}\cdots p_m^{k_m}=x$ is a nontrivial $\Tn$-factorization of $x$. However, this contradicts that the element form of $x$ is irreducible, concluding the proof.\\

($\impliedby$) The proof is analogous to that of the previous argument.
\end{proof}

Let $\G$ be cyclic. Note that for every generator of the group, there is an automorphism determined by where the generators are mapped. Specifically, one can fix a particular generator and obtain all of the automorphisms of $\G$ by sending that generator to the other ones. As seen previously, these automorphisms preserve $\Tn$-atoms. As a consequence of this previous result, we obtain the following:

\begin{Proposition }
\label{bijection}
If $|\c{x}|=|\c{y}|$ under $\G$, then there is a bijection between the irreducible element forms equivalent to $\c{x}$ and the irreducible element forms equivalent to $\c{y}$ in $\G$.
\end{Proposition }
\begin{proof}
Take $\varphi$ to be the automorphism of $\G$ which maps $\c{x}$ to $\c{y}$. Define $I_z$ be the set of all irreducible element forms equivalent to $\c{z}$ for that $n$. Consider the map $f:I_x\rightarrow I_y$, given by $(\c{x_1})^{m_1}\cdots(\c{x_k})^{m_k}\mapsto \varphi(\c{x_1})^{m_1}\cdots\varphi(\c{x_k})^{m_k}$. Note that $f$ is well defined because $\varphi$ maps elements equal to $\c{x}$ to elements equal to $\c{y}$, $\varphi$ is well defined, and by Proposition \ref{auto}. By the definition of element forms, two element forms are different if and only if their exponents for $\c{x_1},\ldots,\c{x_k}$ do not match. Since $f$ preserves these exponents and $\varphi$ is injective, the images of different element forms have to be different. This shows that $f$ is an injective function. On the other hand, let $(\c{y_1})^{m_1}\cdots(\c{y_k})^{m_k}\in I_y$. Since $\varphi$ is a bijection, there exist $x_1,\ldots, x_k$ such that $\c{y_i}=\varphi(\c{x_i})$. Since $\c{y}=\c{y_1^{m_1}\cdots y_k^{m_k}}$, and by the properties of automorphisms, we have that $\c{y}=\varphi(\c{x_1^{m_1}\cdots x_k^{m_k}})$ so that $\c{x_1^{m_1}\cdots x_k^{m_k}}=\c{x}$. The element form $(\c{x_1})^{m_1}\cdots(\c{x_k})^{m_k}$ is irreducible by Proposition \ref{auto}. And so, $f((\c{x_1})^{m_1}\cdots(\c{x_k})^{m_k})=(\c{y_1})^{m_1}\cdots(\c{y_k})^{m_k}$, proving that $f$ is surjective.  Therefore, $f$ provides a bijection from $I_x$ to $I_y$, as we wanted to show.
\end{proof}

The previous proposition says that if we find all the irreducible element forms equivalent to $\c{x}$, and $|\c{x}|=m$, then we effectively also have all the irreducible element forms for any $y\in\G$, that satisfies $|\c{x}|=|\c{y}|$, thus reducing the number of cases to check in order to find all $\Tn$-irreducible elements.\\

The next proposition shows that under certain conditions we can ignore primes, $p$, such that $\c{1}=\c{p}$ in an element form, when considering whether or not it is irreducible.

\begin{Proposition }
\label{noone}
Let the element form $(\c{x_2})^{m_2}\cdots(\c{x_k})^{m_k}$ not be equivalent to $\c{1}$. Then, $(\c{x_2})^{m_2}\cdots(\c{x_k})^{m_k}$ is irreducible if and only if $(\c{x_1})^m\cdot(\c{x_2})^{m_2}\cdots(\c{x_k})^{m_k}$ is irreducible for all $m$.
\end{Proposition }
\begin{proof}
Let $p_i\in\c{x_i}$ for all $i$. We shall prove both sides by contrapositive.\\

($\implies$) Let $p_1^m p_2^{m_2}\cdots p_k^{m_k}$ have a nontrivial $\Tnp$-factorization $a_1*a_2***a_r$. Define $a_i'=a_i/p_1^c$, where $c$ is the largest power of $p_1$ that divides $a_i$. Note that $a_1'\cdots a_r'=p_2^{m_2}\cdots p_k^{m_k}$ and $\c{a_i}=\c{a_i'}$ for all $i$. Therefore, $a_1'*a_2'***a_r'$. Note that $a_i,a_i'\not\equiv\pm 1\pmod n$ for all $i$, because we imposed the condition that the element form $(\c{x_2})^{m_2}\cdots(\c{x_k})^{m_k}$ not be equivalent to $\c{1}$. Therefore, $(\c{x_2})^{m_2}\cdots(\c{x_k})^{m_k}$ is not irreducible, as we wanted to show.\\

($\impliedby$) Let $m\in\Z^+$. Now let $p_2^{m_2}\cdots p_k^{m_k}$ have a nontrivial $\Tnp$-factorization, $a_1*a_2***a_r$. We now have that $(p_1^m\cdot a_1)*a_2***a_r$ so that $(\c{x_1})^m(\c{x_2})^{m_2}\cdots(\c{x_k})^{m_k}$ is irreducible for all $m$.
\end{proof}

For the next definition keep in mind that for any element form $(\c{x_1})^{m_1}\cdots(\c{x_k})^{m_k}$, it is true that $\c{x_1}=\c{1}$, because $\gcd(n,1)=1$, and because we choose that $x_1<x_2<\cdots<x_k$.

\begin{definition}
\label{subs}
Let $\alpha=(\c{x_1})^{m_1}\cdots(\c{x_k})^{m_k}$ be an element form. Let  $\c{x_i}=\c{x_r}\cdot\c{x_s}$, for $\c{x_r},\c{x_s}\neq\c{1}$ for some $i$. We say that an element form $\beta=(\c{x_1})^{l_1}\cdots (\c{x_k})^{l_k}$ is obtained by a single substitution on $\alpha$ in $\c{x_i}$ if one of the following conditions holds:

\begin{enumerate}
\item $m_i\geq 1$; $l_i=m_i-1$; $l_r=m_r+1$; $l_s=m_s+1$; and $l_j=m_j$ for all other values of $j$, whenever $s\neq r$.
\item $m_i\geq 1$; $l_i=m_i-1$; $l_r=l_s=m_r+2$; and $l_j=m_j$ for all other values of $j$, whenever $s=r$.
\end{enumerate}

We say that $\beta=(\c{x_1})^{l_1}\cdots (\c{x_k})^{l_k}$ is obtained by a double substitution on $\alpha$ if $i=1$ and one of the following conditions holds:

\begin{enumerate}
\item $l_r=m_r+1$; $l_s=m_s+1$; and $l_j=m_j$ for all other values of $j$ (including $j=1$), whenever $s\neq r$.
\item $l_r=l_s=m_r+2$; and $l_j=m_j$ for all other values of $j$ (including $j=1$), whenever $s=r$.
\end{enumerate}

\end{definition}

This definition is just a formalization of saying that $\beta$ is a single substitution of $\alpha$, if we take $x_i$ in $\alpha$ and replace it with $x_r\cdot x_s$ where $\c{x_i}=\c{x_r}\cdot\c{x_s}$, and that $\beta$ is a double substitution of $\alpha$, if we just add on two factors into $\alpha$ whose product is equivalent to $\c{1}$.\\

\begin{definition}
\label{simp}
Let $\alpha=(\c{x_1})^{m_1}\cdots(\c{x_k})^{m_k}$ be an element form. Let  $\c{x_i}=\c{x_r}\cdot\c{x_s}$, for some $i$. We say that an element form $\beta=(\c{x_1})^{l_1}\cdots (\c{x_k})^{l_k}$ is obtained by a single simplification on $\alpha$ of $\c{x_r},\c{x_s}$ if $i\neq 1$ and one of the following conditions holds:

\begin{enumerate}
\item $m_r,m_s\geq 1$; $l_i=m_i+1$; $l_r=m_r-1$; $l_s=m_s-1$; and $l_j=m_j$ for all other values of $j$, whenever $s\neq r$.
\item $m_r\geq 2$; $l_i=m_i+1$; $l_r=l_s=m_r-2$; and $l_j=m_j$ for all other values of $j$, whenever $s=r$.
\end{enumerate}

We say that $\beta=(\c{x_1})^{l_1}\cdots (\c{x_k})^{l_k}$ is obtained by a double simplification on $\alpha$ of $\c{x_r},\c{x_s}$ if $i=1$ and one of the following conditions holds:

\begin{enumerate}
\item $m_r,m_s\geq 1$; $l_r=m_r-1$; $l_s=m_s-1$; and $l_j=m_j$ for all other values of $j$ (including $j=1$), whenever $s\neq r$.
\item $m_r\geq 2$; $l_r=l_s=m_r-2$; and $l_j=m_j$ for all other values of $j$ (including $j=1$), whenever $s=r$.
\end{enumerate}

\end{definition}

Definition \ref{simp} is essentially the inverse (or dual) process of Definition \ref{subs}. In Definition \ref{simp}, when two factors of the form $\beta$ simplify to $\c{1}$, we do not add on the factor of $\c{1}$.


We shall use these tools developed to prove the following proposition:

\begin{Proposition }
\label{reduce}
In $\G$, let $\alpha=(\c{x_1})^{m_1}\cdots(\c{x_k})^{m_k}$ be an irreducible element form that $x$ is an element of. Any simplification (single or double) of $\alpha$ is irreducible. 
\end{Proposition }

\begin{proof}

We shall prove this by contrapositive. Let $p_i\in\c{x_i}$ for all $i$ and let  $(\c{x_1})^{l_1}\cdots(\c{x_k})^{l_k}$ be the simplification of $\alpha$, as in Definition \ref{simp}. Let $p_1^{m_1}\cdots p_k^{m_k}\in\alpha$ and let $p_1^{l_1}\cdots p_k^{l_k}$ be in the simplification. By assumption, the simplification has a nontrivial $\Tnp$-factorization $a_1*a_2***a_t$. We need to check that the proposition holds for all four cases of the definition. We do only the first case for both single and double simplifications (the other cases are similar).\\

For the single simplification case, $p_i\mid a_c$ for some $c$. Now, let $a_j'=a_j$ for $j\neq c$, and $a_c'=a_c\cdot p_r\cdot p_s/p_i$. By Definition \ref{simp}, $a'_c$ is an integer and $\c{a_c'}=\c{a_c}$. Then, $a_1'\cdots a_t'=p_1^{m_1}\cdots p_k^{m_k}$ and $a_1'***a_t'$ so that $\alpha$ is not irreducible. For the double simplification case, let $a_1'=a_1\cdot p_r\cdot p_s$ and $a_j'=a_j$ for all other $j$. We then have that $a_1'\cdots a_t'=p_1^{m_1}\cdots p_k^{m_k}$ and $a_1'***a_t'$ so that $\alpha$ is not irreducible. This concludes the proof.
\end{proof}

\begin{definition}
\label{sequence}
We define
\begin{enumerate}
\item A sequence of element forms is any sequence $\{\alpha_k\}_{k\in\Z^+}$, where $\alpha_n$ is an element form for each $n$.
\item A sequence of irreducible element forms $\{\alpha_k\}$ is such that $\alpha_n$ is irreducible for all $n$. 
\item An element form $\alpha$ can be obtained through a sequence of substitutions, if there exists a sequence of element forms $\{\alpha_k\}_{k=1}^m$ such that $\alpha_1=(\c{a})^1$ for some $a\in\G$, $\alpha_m=\alpha$ and $\alpha_{i+1}$ is obtained through a substitution on $\alpha_i$ for all $0<i\leq m-1$.
\item Conversely,  an element form $\alpha$ can be obtained through a sequence of simplifications, if there exists a sequence of element forms $\{\alpha_k\}_{k=1}^m$ such that $\alpha_m=(\c{a})^1$ for some $a\in\G$, $\alpha_1=\alpha$ and $\alpha_{i+1}$ is obtained through a simplification on $\alpha_i$ for all $0<i\leq m-1$.
\end{enumerate}
\end{definition}

If $\alpha$ is obtained by a sequence of substitutions $\{\alpha_k\}_{k=1}^m$, then $\alpha$ can be obtained through a sequence of simplifications $\{\beta_k\}_{k=1}^m$, where $\beta_k=\alpha_{m-k+1}$ for all $k$.\\

We shall now develop methods for finding irreducible element forms.

\begin{corollary}
\label{allirr}
If $\alpha$ is an irreducible element form, then every sequence of simplifications of $\alpha$, $\{\alpha_k\}_{k=1}^m$, satisfies that $\alpha_k$ is irreducible for all $k$.
\end{corollary}
\begin{proof}
The result is clear from Definition \ref{sequence} and Proposition \ref{reduce}.
\end{proof}

\begin{Proposition }
\label{prealg}
Each irreducible element form equivalent to $\c{x}\in\G$, and with $m_1=0$, can be obtained by a sequence of substitutions $\{\beta_k\}_{k=1}^m$, where $\beta_k$ is irreducible for all $k$.
\end{Proposition }
\begin{proof}

Let  $\alpha=(\c{x_1})^{m_1}\cdots(\c{x_k})^{m_k}$ be an irreducible element form equivalent to $\c{x}\in\Gp$. If there exists any sequence of substitutions for $\alpha$, then the condition that each entry in the sequence is irreducible holds by Corollary \ref{allirr}. If the elements of $\alpha$ are primes, then we are done. Define a sequence of simplifications on $\alpha$ in the following way: if the elements of $\alpha_i$ have more than one prime dividing them, let $\alpha_{i+1}$ be any fixed simplification of $\alpha_i$, (which exists by Definition \ref{simp}). Since the exponents of $\alpha$ are all finite, and $\G$ is finite, we have that this process eventually reaches $\alpha_m=(\c{x})^1$ for some $m$. Finally, we conclude that  $\{\beta_k\}_{k=1}^m$, where $\beta_k=\alpha_{m-k+1}$ for all $k$ is the sequence of substitutions that we were looking for.
\end{proof}

Note that the restriction that $m_1=0$, is due to the definition of substitutions on element forms does not give a way to add on factors equivalent to $\c{1}$. However, Proposition $\ref{noone}$ makes it so that this does not matter.\\

With this result, we can come up with the following preliminary algorithm for finding $\Tn$-irreducible elements.

\begin{algorithm}
To find all irreducible element forms in $\G$,
\begin{enumerate}
\item Fix elements of distinct order $\c{a_1},\ldots, \c{a_k}\in\G$.

\item Set $i=1$

\item Set the element form $(\c{x})^1$ to be equivalent to $\c{a_i}$.

\item Obtain all element forms obtainable by one single or double substitutions on $(\c{x})^1$. Keep the ones that are irreducible and discard all that are not.

\item Repeat this process with the resulting element forms.

\item Return to step (5) until all irreducible element forms equivalent to $\c{a_i}$ are found.

\item Set $i=i+1$ and return to step (3) or move to next step if $i=k$.

\item Apply all automorphisms of $\G$ to all the irreducibles, and keep one copy of each.
\end{enumerate}
\end{algorithm}

If we do not restrict the value of $n$, Proposition \ref{method} may not hold. As a consequence of this, we may never reach step (7) at times. An example of how this algorithm can continue indefinitely is how $(\c{3})^i(\c{5})^1$ is irreducible for all $i$ in $\T_{(13)}$. Being able to find all irreducible element forms with an algorithm that terminates in a finite amount of steps is still an open problem. However, if $n$ is a safe prime associated to a Sophie Germain prime, then Proposition \ref{method} does hold, guaranteeing that this algorithm terminates in a finite number of steps. For general values of $n$, the only way to guarantee that the algorithm theoretically terminates is if the algorithm is executed in a parallel way for all values of $i$. It is specifically because of Proposition \ref{prealg} that we can discard the element forms that are not irreducible in steps (4) and (5), thus making this algorithm finite for safe primes. \\

This method can be further optimized, however. First note that if $\alpha$ is a nontrivial irreducible form equivalent to $\c{1}$, then any double substitution done on $\alpha$ is not irreducible. In this case, doing a double substitution would just add two primes whose product is in $\c{1}$, while the rest of the number (excluding the added primes) is also in $\c{1}$. So, in the case of $\alpha$ equivalent to $\c{1}$, we need only do single substitutions. One could ask the question of whether or not it is possible to obtain all irreducible element forms by doing only single substitutions in general. \\

\begin{Proposition }
Any irreducible element form (without a power of $\c{1}$) equivalent to $\c{x}\in\G$ can be obtained through as a sequence of simple substitutions on the element form $(\c{x})^1$.
\end{Proposition }
\begin{proof}

The case of $\alpha$ being equivalent to $\c{1}$ was already considered. Assume that $\alpha$ is equivalent to $\c{x}\neq\c{1}$. We shall prove the result by induction on the number of primes that divide the elements of $\alpha$. If $a=1$ or $a=2$, then the result is clear. Assume that for all $1\leq a<l$ that the result is true, and we shall prove it for $l$. Let $\{\alpha_k\}_{k=1}^m$ be a sequence of substitutions of element forms that ends in $\alpha$, which exists, by Proposition \ref{prealg}. If $\alpha_m=\alpha$ is obtained by a single substitution on $\alpha_{m-1}$, then we are done. This is because, by the inductive hypothesis we have that the number of primes that divide the elements of $\alpha_{m-1}$ is lesser than $l$, meaning that there exists a sequence of single substitutions ending in $\alpha_{m-1}$, to which we can append the single substitution that makes $\alpha$. Therefore, we can assume that for every $\alpha$ that has elements with exactly $l$ primes dividing it, that for every sequence of substitutions that ends in $\alpha$ it is true that the last substitution in the sequence is a double substitution. Let $\{\alpha_k\}_{k=1}^m$ be one such sequence. Let $x=p_2^{m_1}\cdots p_k^{m_k}$ be an element of $\alpha$. If there existed $p_i,p_j$ such that $m_i,m_j\geq 1$ that satisfy $\c{p_i\cdot p_j}\neq\c{1}$, then the single simplification on $\alpha$ by putting together $\c{x_i}$ and $\c{x_j}$ is a single simplification, and the resulting element form would have a sequence of single substitutions producing it by the inductive hypothesis, a contradiction. Therefore, for all $p_i,p_j$ it is true that $\c{p_i\cdot p_j}=\c{1}$. In other words, $x$ would have at most two different primes dividing it. If there were at least three different primes dividing $x$, $p_r,p_s,p_t$, then the product of at least two of them would not be the identity in $\Gp$. It is then true that if $x=p_r^bp_s^c$, then $|b-c|=1$, since otherwise it would contradict that $\alpha$ is an irreducible element form. At least one of $b$ or $c$ has to be greater than $1$ in order for $l>2$. Additionally, $b,c\neq 0$, since otherwise $\alpha$ would not be an irreducible element form. Without loss of generality, let $b>1$. All of these conditions would imply that $\c{p_r^2}=\c{1}$, meaning that $\c{p_r}$ is its own inverse. However, $\c{p_s}$ is the inverse of $\c{p_r}$. This implies that $r=s$, a contradiction to $\alpha$ being an irreducible element form.\\

Since we considered all possible cases, this concludes the proof.
\end{proof}
With this, we can eliminate all the cases in which we would attempt a double substitution in the algorithm (which in practice saves about half the run time). Modifying the algorithm, we obtain:

\begin{algorithm}
\label{alg}
To find all irreducible element forms in $\G$,
\begin{enumerate}
\item Fix elements of distinct order $\c{a_1},\ldots, \c{a_k}\in\G$.

\item Set $i=1$

\item Set the element form $(\c{x})^1$ to be equivalent to $\c{a_i}$.

\item Obtain all element forms obtainable by one single substitution on $(\c{x})^1$. Keep the ones that are irreducible and discard that are not.

\item Repeat this process with the resulting element forms.

\item Return to step (5) until all irreducible element forms equivalent to $\c{a_i}$ are found.

\item Set $i=i+1$ and return to step (3) or move to next step if $i=k$.

\item Apply all automorphisms of $\G$ to all the irreducibles, and keep one copy of each.
\end{enumerate}
\end{algorithm}

The same comments that were said for the previous algorithm apply for this algorithm, especially the observation that if $n$ is not a safe prime, then we cannot guarantee that this algorithm will terminate in a finite number of steps.\\

We shall give an example illustrating this algorithm for $n=11$. In this case, $\G$ has $5$ elements, $\c{1},\c{2},\c{3},\c{4},\c{5}$. Observe the Cayley table for $\G$ in Table \ref{table:1}. Since all the elements in any element form are all irreducible or are all not irreducible, it suffices to work with one representative of each element different from $\c{1}$, so that in this specific case, we can consider $23\in\c{1},2\in\c{2}$, $3\in\c{3}$, $7\in\c{4}$, and $5\in\c{5}$. We can represent the algorithm as the following graph, for $a_2$:

$$\scalebox{1}
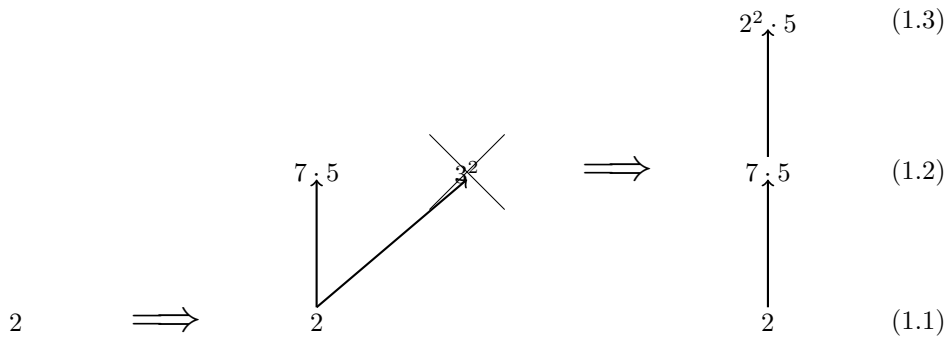
{
\begin{tikzpicture}
   \node at (12, 0) {$(1.1)$};
   \node at (12, 2) {$(1.2)$};
   \node at (12, 4) {$(1.3)$};

   \node at (0, 0) {$2$};
   \node at (2,0) {$\mathlarger{\mathlarger{\mathlarger{\mathlarger{\Longrightarrow}}}}$};
   \node at (8,2) {$\mathlarger{\mathlarger{\mathlarger{\mathlarger{\Longrightarrow}}}}$};
   \node at (4, 0) {$2$};
   
   \node at (4, 2) {$7\cdot 5$};
   \node at (6, 2) {$3^2$};
   
   \draw (5.5,1.5) -- (6.5,2.5);
   \draw (5.5,2.5) -- (6.5,1.5);
   
   \draw[thick,->] (4,0.2) -- (4,1.9);
   
   \draw[thick,->] (4,0.2) -- (6,1.9);
   \node at (10, 0) {$2$};
   
   \node at (10, 2) {$7\cdot 5$};
   
   \node at (10, 4) {$2^2\cdot 5$};
   
   \draw[thick,->] (10,0.2) -- (10,1.9);
   \draw[thick,->] (10,2.2) -- (10,3.9);
\end{tikzpicture}}$$

Note that at each level, every product has the same number of factors as the number of the level it is in. So we start with a single node, as in (1.1). We then do all simple substitutions, which can be seen in (1.2), in the second step. Now, with this level, (1.2), we see that the rightmost number turns out to factorize. Thus, we cross it out as it does not produce any irreducible elements. The next level produces $4$ terms (1.3). Only one of them is an irreducible element. Finally, doing substitutions on the remaining term produces no irreducible elements. This concludes the process for $a_2$, by the previous discussion. When the correspondence is made from the previous graph to the element forms, and all the automorphisms $\varphi$ are applied, we obtain the following graph:

$$\scalebox{1}
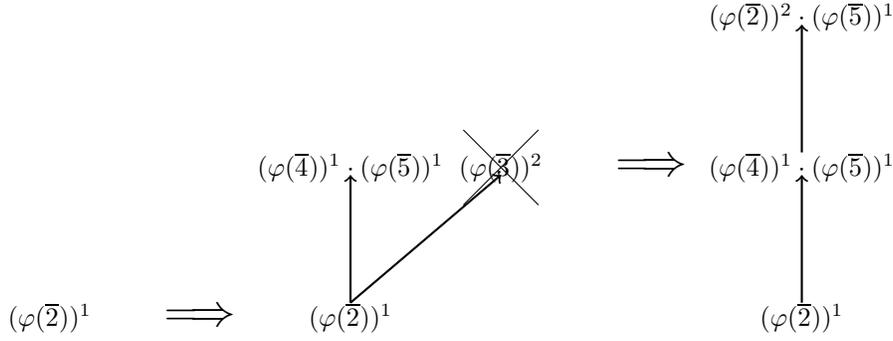
{
\begin{tikzpicture}

   \node at (0, 0) {$(\varphi(\c{2}))^1$};
   \node at (2,0) {$\mathlarger{\mathlarger{\mathlarger{\mathlarger{\Longrightarrow}}}}$};
   \node at (8,2) {$\mathlarger{\mathlarger{\mathlarger{\mathlarger{\Longrightarrow}}}}$};
   \node at (4, 0) {$(\varphi(\c{2}))^1$};
   
   \node at (4, 2) {$(\varphi(\c{4}))^1\cdot (\varphi(\c{5}))^1$};
   \node at (6, 2) {$(\varphi(\c{3}))^2$};
   
   \draw (5.5,1.5) -- (6.5,2.5);
   \draw (5.5,2.5) -- (6.5,1.5);
   
   \draw[thick,->] (4,0.2) -- (4,1.9);
   
   \draw[thick,->] (4,0.2) -- (6,1.9);
   \node at (10, 0) {$(\varphi(\c{2}))^1$};
   
   \node at (10, 2) {$(\varphi(\c{4}))^1\cdot (\varphi(\c{5}))^1$};
   
   \node at (10, 4) {$(\varphi(\c{2}))^2\cdot (\varphi(\c{5}))^1$};
   
   \draw[thick,->] (10,0.2) -- (10,1.9);
   \draw[thick,->] (10,2.2) -- (10,3.9);
\end{tikzpicture}}$$


There are three automorphisms other than the identity for $U'(11)$: the one that sends $2$ to $3$, the one that sends $2$ to $4$ and the one that sends $2$ to $5$. Therefore, by Proposition \ref{bijection}, the previous graph contains all the element forms equivalent to $a_2,a_3,a_4$, and $a_5$. Explicitly, the irreducible element forms are $(\c{1})^k(\c{5})^1$, $(\c{1})^k(\c{4})^1$, $(\c{1})^k(\c{3})^1$, $(\c{1})^k(\c{2})^1$, $(\c{1})^k(\c{2})^1(\c{3})^1$, $(\c{1})^k(\c{2})^1(\c{4})^1$, $(\c{1})^k(\c{3})^1(\c{5})^1$, $(\c{1})^k(\c{4})^1(\c{5})^1$, $(\c{1})^k(\c{2})^2(\c{5})^1$, $(\c{1})^k(\c{3})^2(\c{4})^1$, $(\c{1})^k(\c{3})^1(\c{4})^2$, $(\c{1})^k(\c{2})^1(\c{5})^2$.\\

Similarly, we can do this graph for $a_1$. Again, discarding all the elements that factorize, (3 in the fourth level and 12 in the fifth):

$$\scalebox{1}
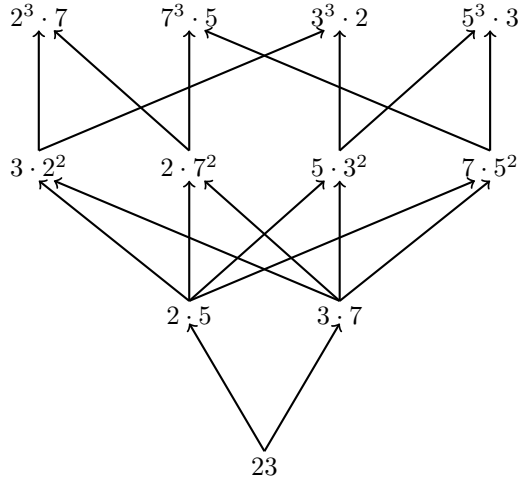
{
\begin{tikzpicture}
   
   \node at (0, 0) {$23$};
   
   \node at (-1, 2) {$2\cdot 5$};
   \node at (1, 2) {$3\cdot 7$};

   \node at (-3, 4) {$3\cdot 2^2$};
   \node at (-1, 4) {$2\cdot 7^2$};
   \node at (1, 4) {$5\cdot 3^2$};
   \node at (3, 4) {$7\cdot 5^2$};
   
   \node at (-3, 6) {$2^3\cdot 7$};
   \node at (-1, 6) {$7^3\cdot 5$};
   \node at (1, 6) {$3^3\cdot 2$};
   \node at (3, 6) {$5^3\cdot 3$};
   
   \draw[thick,->] (0,0.2) -- (1,1.9);
   \draw[thick,->] (0,0.2) -- (-1,1.9);
   
   \draw[thick,->] (-1,2.2) -- (2.8,3.8);
   \draw[thick,->] (1,2.2) -- (-2.8,3.8);
   \draw[thick,->] (1,2.2) -- (3,3.8);
   \draw[thick,->] (-1,2.2) -- (-3,3.8);   
   \draw[thick,->] (-1,2.2) -- (0.8,3.8);
   \draw[thick,->] (1,2.2) -- (-0.8,3.8);
   \draw[thick,->] (1,2.2) -- (1,3.8);
   \draw[thick,->] (-1,2.2) -- (-1,3.8);   
   
   \draw[thick,->] (1,4.2) -- (1,5.8);
   \draw[thick,->] (-1,4.2) -- (-1,5.8);  
   \draw[thick,->] (3,4.2) -- (3,5.8);
   \draw[thick,->] (-3,4.2) -- (-3,5.8);  
   \draw[thick,->] (1,4.2) -- (2.8,5.8);
   \draw[thick,->] (-1,4.2) -- (-2.8,5.8);  
   \draw[thick,->] (3,4.2) -- (-0.8,5.8);
   \draw[thick,->] (-3,4.2) -- (0.8,5.8);  
   
\end{tikzpicture}}$$

Explicitly, the irreducible element forms for this graph are: $(\c{1})^1$, $(\c{2})^1(\c{5})^1$, $(\c{3})^1(\c{4})^1$, $(\c{2})^2(\c{3})^1$, $(\c{2})^1(\c{4})^2$, $(\c{3})^2(\c{5})^1$, $(\c{4})^1(\c{5})^2$, $(\c{2})^3(\c{4})^1$, $(\c{4})^3(\c{5})^1$, $(\c{2})^1(\c{3})^3$, $(\c{3})^1(\c{5})^3$. The important thing to keep in mind is that the irreducible element forms are easily recoverable from this graph. We conclude that the numbers that remain in the graphs uniquely induce the element forms containing all irreducible elements in $U'(11)$, for a total of $3\cdot 4+11=23$ irreducible element forms.\\

There are two final things that we are going to take note on, with respect to the implementation of Algorithm \ref{alg}. Firstly, for any given element $\c{x}\in\G$, how do we know the values of $y,z$ such that $\c{x}=\c{y}\cdot\c{z}$? Secondly, how do we know that a specific element form is not irreducible, so that it can be discarded as in steps (4) and (5) of the algorithm? Both of these questions are crucial to the algorithm in steps (4) and (5). We have seen throughout the paper that which element forms are irreducible or not depends almost exclusively on the group structure of $\G$. For example, since $U'(7)\cong U'(9)$, it is easy to see that the irreducible element forms for both $n=7$ and $n=9$ are in bijection. By considering all different values of $n$ individually, there is bound to be a lot of redundancy. In order to avoid this, we could instead work with the more familiar finite abelian groups that $\G$ is isomorphic to. For example, since $U'(7),U'(9)\cong\Zn{3}$, we could extend the rules for factorization to $\Zn{3}$ and work there alone. The elements of $\Zn{3}$ are $\c{0},\c{1},\c{2}$. The rule for factorization is that if you have some sum involving the three elements of $\Zn{3}$, and the sum can be simplified in some way which makes all new terms (there have to be at least 2 terms in the simplification) the same, then a number is said to factor. This same rule can be extended to any finite abelian group, including non cyclic groups like $\Zn{2}\times\Zn{2}$, which are necessary for the study of $\G$ for values like $n=24$.\\

For the first question, a naïve implementation could involve making the Cayley table for each $\G$, and then making separate lists for each element, which tell what are the possible single substitutions. However, if $n$ is a value for which $\G\cong\Zn{m}$, then the previous discussion suggests that we can run the algorithm on $\Zn{m}$ instead. The process of generating numbers $\c{y},\c{x}$ such that $\c{z}=\c{x}+\c{y}$ is less complicated, as the list

$$\c{1}+\c{z-1},\c{2}+\c{z-2},\cdots,\c{m-1}+\c{z-m+1}$$

contains all relevant values of $\c{x}$ and $\c{y}$ with multiplicity $2$. So, this solves the first question, especially when $\G$ is cyclic.\\

For the second question, consider the number $x$, which belongs to an element form generated by the algorithm. If we wanted to see whether or not this number has a $\Tn$-factorization, a brute force method could be to generate all possible factorizations of $x$ and then checking, one by one, if any of the factorizations is also a $\Tn$-factorization. To illustrate why this is difficult and redundant, consider $p^k$. The number of factorizations that $p^k$ has in $\Z$ is the number of partitions of $k$. Since the number of partitions grows incredibly fast as $k$ is bigger, this would prove difficult to do for big values of $k$. The situation for a general number $x$ is even worse. So, consider the following method. Assume that we are working in $\Zn{m}$ (this method can be extended to other groups, however). Let us say that we are considering an element form $\alpha$ in the $m$-th level of the tree produced by the algorithm, and we do not know if $\alpha$ is irreducible or not. Since $\alpha$ is in the $l$-th level, this means that $\alpha$ is a sum with exactly $l$ terms. First consider the case of when $\alpha$ is not equivalent to $\c{0}$. Let $\alpha$ not be the multiple of some element of $\Zn{m}$ (this is easy to detect). We then have that if $\alpha$ is not irreducible, then there exists some single or double simplification of $\alpha$ that is not irreducible either. By contradiction, assume this to be false. Since $\alpha$ is not a multiple of some fixed element, it would take at least one simplification on $\alpha$ before the resulting form is the multiple of some element. In particular, if any simplification of $\alpha$ is taken with respect to two element $\c{y},\c{z}\in\Zn{m}$ which appear grouped together in any factorization of $\alpha$, then such element form has to be not irreducible as well, a contradiction. This fails in the case that $\alpha$ is equivalent to $\c{0}$ because in the definition of simplification, we discard factors of $\c{0}$ that arise in any simplification. Thus, if $\alpha$ is a sum of length $l\geq 3$ (for fixed $l$), the only other extra condition that has to be imposed for $\alpha$ to be irreducible is that no two terms in $\alpha$ sum to $\c{0}$.\\

To summarize, a method of telling, in steps (4) and (5) of the algorithm, whether or not an element form, $\alpha$, is irreducible or not, we only have to check three conditions: if $\alpha$ is equivalent to $\c{0}$ and is a sum of at least $3$ terms, that there do not exist two terms that sum to $\c{0}$, that any one simplification done on $\alpha$ is in the tree produced by the algorithm, and that $\alpha$ is not the multiple of some element in $\Zn{m}$. While this seems more convoluted, in the worst of cases (when every term in $\alpha$ is different), if $\alpha$ has $l$ terms, one would only have to check $O(l^2)$ cases.\\

With all this in mind, the following table summarizes the number of irreducible element forms for $\Zn{m}$, for the values of $m$ that have been completely classified:

\begin{table}[h]
\scalebox{1}{
\begin{tabular}{|c||c|c|}
 \hline
 $m$ & Number of irreducible element forms for $\Zn{m}$& Number of values $n$ such that $U'(n)\cong\Zn{m}$ \\ 
 \hline
 \hline
$1$ & 1 & 3\\
\hline
$2$ & 2 & 4\\
\hline
$3$ & 4 & 4\\
\hline
$4$ & $7^*$ (infinitely many)& 4\\
\hline
$5$ & 23& 2\\
\hline
$6$ & $124^*$ (infinitely many)& 6 \\
\hline
$7$ & 108& 0\\
\hline
$11$ & 1398& 2\\
\hline
$13$ & 4367& 0\\
\hline
$17$ & 33321& 0\\
\hline
$19$ & 84544& 0\\
\hline
$23$ & 465774& 2\\
\hline
\end{tabular}}
\caption{Algorithm results summary}
\label{table:2}
\end{table}

Values with asterisks ($*$) imply that the number represents the number of element forms in which any fixed term is repeated less times that the term's order. In general, all irreducible elements that do not satisfy this condition can be determined from the element forms that satisfy the previous condition. The precise reasons for this are outside the scope of this paper. However, the methods known to determine all the irreducible element forms for these values of $n$ (i.e. $m=6$) require checking cases by hand. It is for this reason that composite values of $m$ have not been explored much. On the other hand, to see why the third column of the table is $0$ sometimes, it is enough to use Proposition \ref{Gprime}. This proposition tells us that almost all the primes for which there is a corresponding $U'(n)$ are Sophie Germain primes. In particular, $7,13,17$ and $19$ are not Sophie Germain primes. These values were chosen because the algorithm terminates in a finite number of steps, and is able to be executed without any human intervention.\\

As a final note, if $U'(p)\cong\Zn{m}$ and $m$ appears on the table, then for those values of $p$ we have found all $\Tp$-atoms. It is not too difficult to check that, if $x=pt$, then $x$ is a $\Tp$-atom if and only if $p\nmid t$. This would mean that all $\Tp$-atoms for these cases would be the atoms referred to in the table along with numbers $pt$ with $p\nmid t$.

\section{Conclusion and Future Works}\label{concl}

As seen from the results, we were able to find exactly when $\G$ is cyclic. The general group structure for $\G$ is the following: 

\begin{Proposition }
Let $2^k$ be the minimal order of the cyclic groups that compose $\syl{2}{U(n)}$. If $U(n)\cong F\o B\o\Zn{2^k}$ where $F$ has odd order, and $B\o\Zn{2^k}\cong\syl{2}{U(n)}$, then $U'(n)\cong F\o B\o\Zn{2^{k-1}}$.
\end{Proposition }

However, the proof is outside the scope of this paper. Additionally, we were able to determine an algorithm which theoretically finds all $\Tn$-irreducible elements that are relatively prime to $n$. If $n=p$, this means that the algorithm theoretically finds all $\Tp$-atoms. In particular, if $p$ is a safe prime associated to a Sophie Germain prime, then this algorithm terminates in a finite number of steps. Previous to this research, the $\Tn$-atoms had only been known for $n=1,2,3,4,5,6,8,10$ and $n=12$, which correspond to $\Zn{1}$ and $\Zn{2}$ in Table \ref{table:2}. Other than for these values, there only existed speculations for the list of all $\Tn$-atoms for $n=7$ and $n=11$. In this research the irreducible element forms for $\Zn{m}$ were additionally found for $m=3,5,7,11,13,17,19$ and $23$, which correspond to $10$ new values of $n$ such that we either know all or a significant portion of the $\Tn$-atoms. For $m=4,6$ all irreducible element forms are known for $\Zn{m}$, even though human aid was required. For these two values there correspond another $10$ values of $n$ for which a significant amount (if not all) of the $\Tn$ are acounted for.\\

On ther other hand, little is know about the $\Tn$-atoms that are not relatively prime to $n$. Prior to this work, $\Tn$-atoms were known for $n=4,8,12$, and the methods used could help shed light on this problem. Complete work has been done for $n=9, 14, 15, 16, 18, 22, 25, 27, 32, 49$ and $n=121$, with tools that are out of the scope of this paper. Despite this, the irreducible element forms for many of these values are still unknown, particularly $n=25, 27, 32$ and $121$. The work done for these values has not resulted in an algorithm that can be implimented as easily as the one in this paper, and thus requires a lot of human calculation and a different approach. So a possible future work would try to develop such algorithm for the $\Tn$-atoms that are not relatively prime to $n$.\\

The findings of this paper are relevant for the study of $\Tn$-factors and $\Tn$-graphs. Additionally, due to how the algorithm ends in finitely many steps for almost exclusively safe primes, there might be applications with anything that involves Sophie Germain primes.

\bibliography{paper5}{}
\bibliographystyle{plain}
\end{document}